\documentclass[a4paper]{amsart}
\usepackage[colorlinks,citecolor=blue,urlcolor=black,linkcolor=black]{hyperref}
\usepackage{amssymb,cmap}

\newtheorem*{thma}{Main Theorem}
\newtheorem*{ucor}{Corollary}
\newtheorem{thm}{Theorem}[section]
\newtheorem{lemma}[thm]{Lemma}
\newtheorem{cor}[thm]{Corollary}
\newtheorem{claim}{Claim}[thm]
\newtheorem{prop}[thm]{Proposition}
\newtheorem{fact}[thm]{Fact}

\theoremstyle{definition}
\newtheorem{defn}[thm]{Definition}

\theoremstyle{remark}
\newtheorem{remark}[thm]{Remark}

\DeclareMathOperator{\cf}{cf}
\DeclareMathOperator{\otp}{otp}
\DeclareMathOperator{\acc}{acc}
\DeclareMathOperator{\nacc}{nacc}

\newcommand\s{\subseteq}
\newcommand\sq{\sqsubseteq}
\newcommand\br{\blacktriangleright}
\newcommand*\axiomfont[1]{\textsf{\textup{#1}}}
\newcommand\zfc{\axiomfont{ZFC}}
\newcommand\gch{\axiomfont{GCH}}
\newcommand\Mid{\Biggm|}
\renewcommand\mid{\mathrel{|}\allowbreak}

\title{Souslin trees at successors of regular cardinals}

\author{Assaf Rinot}
\address{Department of Mathematics, Bar-Ilan University, Ramat-Gan 5290002, Israel.}

\thanks{This research was partially supported by the European Research Council (grant agreement ERC-2018-StG 802756) and by the Israel Science Foundation (grant agreement 2066/18)}


\begin{document}
\begin{abstract} We present a weak sufficient condition for the existence of Souslin trees at successor of regular cardinals.
The result is optimal and simultaneously improves an old theorem of Gregory and a more recent theorem of the author.
\end{abstract}

\maketitle
\section*{Introduction}

In \cite{MR485361}, Gregory proved that for every (regular) uncountable cardinal $\lambda=\lambda^{<\lambda}$,
if $2^\lambda=\lambda^+$ and there exists a non-reflecting stationary subset of $E^{\lambda^+}_{<\lambda}$,
then there exists a $\lambda^+$-Souslin tree.
A special case of a result from \cite{paper24} asserts that for every uncountable cardinal $\lambda=\lambda^{<\lambda}$,
if $2^\lambda=\lambda^+$ and $\square(\lambda^+)$ holds, then there exists a $\lambda^+$-Souslin tree.
By results from inner model theory, Gregory's theorem implies that if $\gch$ holds,
and there are no $\aleph_2$-Souslin trees, then $\aleph_2$ is a Mahlo cardinal in $\mathrm{L}$,
and our theorem implies that if $\gch$ holds,
and there are no $\aleph_2$-Souslin trees, then $\aleph_2$ is a weakly compact cardinal in $\mathrm{L}$.
While the former corollary follows from the latter, the combinatorial theorem of Gregory does not follow from ours.
The purpose of this note is to present a new combinatorial theorem that implies both:

\begin{thma}
Suppose that $\lambda=\lambda^{<\lambda}$ is an uncountable cardinal, and $2^\lambda=\lambda^+$.

If there exists a $\square(\lambda^+,\lambda)$-sequence $\langle \mathcal C_\alpha\mid\alpha<\lambda^+\rangle$
for which $\{ \alpha\in E^{\lambda^+}_{<\lambda}\mid |\mathcal C_\alpha|<\lambda\}$ is stationary,
then there exists a $\lambda^+$-Souslin tree.
\end{thma}

An immediate corollary to the Main Theorem is an optimal improvement to a result from \cite{paper24} that was promised in \cite{paper29}:
\begin{ucor} Suppose that $\lambda$ is a regular uncountable cardinal.

If $\gch$ and $\square(\lambda^+,{<}\lambda)$ both hold,
then there exists a $\lambda^+$-Souslin tree.
\end{ucor}

The corollary is indeed optimal, since $\gch$ implies that $\square(\lambda^+,\lambda)$ holds for every regular cardinal $\lambda$ (in fact, with a witnessing
sequence $\langle \mathcal C_\alpha\mid\alpha<\lambda^+\rangle$
satisfying $|\mathcal C_\alpha|=1$ for all $\alpha\in E^{\lambda^+}_\lambda$),
whereas by a recent striking result of Asper\'o and Golshani \cite{arxiv1809.07638}, $\zfc+\gch$ is
consistent with the non-existence of a $\lambda^+$-Souslin tree for any prescribed value of a regular uncountable $\lambda$.

\subsection*{Notation and conventions} Throughout this note,
$\kappa$ and $\lambda$ stand for arbitrary regular uncountable cardinals.
Write $[\kappa]^{<\lambda}$ for the collection of all subsets of $\kappa$ of cardinality less than $\lambda$.
Denote $E^\kappa_\lambda:=\{\alpha<\kappa\mid \cf(\alpha)=\lambda\}$
and $E^\kappa_{<\lambda}:=\{\alpha<\kappa\mid \cf(\alpha)<\lambda\}$.

Suppose that $C$ and $D$ are sets of ordinals.
Write $C\sq D$ iff there exists some ordinal $\beta$ such that $C = D \cap \beta$.
Write $\acc(C):=\{\alpha\in C\mid \sup (C\cap\alpha) = \alpha>0 \}$, $\nacc(C) := C \setminus \acc(C)$,
and $\acc^+(C) := \{\alpha<\sup(C)\mid \sup (C\cap\alpha) = \alpha>0 \}$.

\section{Square principles and Souslin trees}

\begin{defn}
For any cardinal $\mu$, $\square(\kappa,{<}\mu)$ asserts
the existence of a sequence $\langle \mathcal C_\alpha\mid\alpha<\kappa\rangle$ such that:
\begin{enumerate}
\item For every limit ordinal $\alpha<\kappa$:
\begin{itemize}
\item$\mathcal C_\alpha$ is a nonempty collection of club subsets of $\alpha$, with $| \mathcal C_\alpha |<\mu$;
\item for every $C \in \mathcal C_\alpha$ and $\bar\alpha\in\acc(C)$, we have $C\cap\bar\alpha\in\mathcal C_{\bar\alpha}$;
\end{itemize}
\item For every club $D\s\kappa$, there is $\alpha\in\acc(D)$ such that $D\cap\alpha\notin\mathcal C_\alpha$.
\end{enumerate}
\end{defn}

\begin{remark}\begin{enumerate}
\item Note that there are no restrictions on $\otp(C)$ for $C\in\mathcal C_\alpha$.
\item We write $\square(\kappa,\mu)$ for $\square(\kappa,{<}\mu^+)$, and write $\square(\kappa)$ for $\square(\kappa,1)$.
\end{enumerate}
\end{remark}

To connect Gregory's theorem with the Main Theorem, let us point out the following.

\begin{prop} Suppose that $\lambda^{<\lambda}=\lambda$ and there exists a non-reflecting stationary subset of $E^{\lambda^+}_{<\lambda}$.
Then there exists a $\square(\lambda^+,\lambda)$-sequence $\langle \mathcal C_\alpha\mid\alpha<\lambda^+\rangle$
for which $\{ \alpha\in E^{\lambda^+}_{<\lambda}\mid |\mathcal C_\alpha|<\lambda\}$ is stationary.
\end{prop}
\begin{proof} Fix a subset $S\s E^{\lambda^+}_{<\lambda}$ which is stationary and non-reflecting.
We now define $\vec{\mathcal C}:=\langle \mathcal C_\alpha\mid\alpha<\lambda^+\rangle$, as follows:

$\br$ Let $\mathcal C_0:=\{\emptyset\}$.

$\br$ For all $\alpha<\lambda^+$, let $\mathcal C_{\alpha+1}:=\{\{\alpha\}\}$.

$\br$ For all $\alpha\in S\cup E^{\lambda^+}_\lambda$, since $S$ is non-reflecting, we may fix a club $C_\alpha$ in $\alpha$ of order-type $\cf(\alpha)$ which is disjoint from $S$.
Now, let $\mathcal C_\alpha:=\{C_\alpha\}$.

$\br$ For all $\alpha\in\acc(\lambda^+)\setminus(S\cup E^{\lambda^+}_{\lambda})$, let $\mathcal C_\alpha$ be the collection of all clubs $C$ in $\alpha$
such that $\otp(C)<\lambda$ and $C\cap S=\emptyset$. As $\cf(\alpha)<\lambda$ and as $S$ is non-reflecting, we know that $\mathcal C_\alpha$
is nonempty. As $\lambda^{<\lambda}=\lambda$, we also know that $|\mathcal C_\alpha|\le\lambda$.

Let us verify that $\vec{\mathcal C}$ is as sought:
\begin{itemize}
\item Evidently, $\{ \alpha\in E^{\lambda^+}_{<\lambda}\mid |\mathcal C_\alpha|<\lambda\}$ covers the stationary set $S$.
\item Fix arbitrary $\alpha\in\acc(\lambda^+)$,  $C\in\mathcal C_\alpha$ and $\bar\alpha\in\acc(C)$. There are two options:

$\br$ If $\alpha\in S\cup E^{\lambda^+}_\lambda$, then $\cf(\bar\alpha)\le\otp(C\cap\bar\alpha)<\otp(C)=\cf(\alpha)\le \lambda$.
Also, $C=C_\alpha$ is disjoint from $S$, so that, altogether, $\bar\alpha\in\acc(\lambda^+)\setminus(S\cup E^{\lambda^+}_\lambda)$.
It now follows from the definition of $\mathcal C_{\bar\alpha}$ that  $C\cap\bar\alpha\in\mathcal C_{\bar\alpha}$.

$\br$ Otherwise, $C$ is a club $\alpha$ such that
$\cf(\bar\alpha)\le\otp(C\cap\bar\alpha)<\otp(C)<\lambda$ and $C\cap S=\emptyset$.
It again follows from the definition of $\mathcal C_{\bar\alpha}$ that  $C\cap\bar\alpha\in\mathcal C_{\bar\alpha}$.

\item Given any club $D$ in $\lambda^+$, pick $\alpha\in D$ such that $\otp(D\cap\alpha)=\lambda+\omega$. Then $\alpha\in\acc(D)$ and $D\cap\alpha\not\in\mathcal C_\alpha$.\qedhere
\end{itemize}
\end{proof}

As mentioned earlier, even in the presence of $\gch$, $\square(\kappa,{<}\kappa)$ does not imply the existence of a $\kappa$-Souslin tree.
For this, Brodsky and the author have introduced the following slight strengthening of $\square(\kappa,{<}\kappa)$:

\begin{defn}[\cite{paper26}]
$\boxtimes^*(\kappa)$ asserts the existence of a sequence $\langle \mathcal C_\alpha\mid\alpha<\kappa\rangle$ such that:
\begin{enumerate}
\item For every limit ordinal $\alpha<\kappa$:
\begin{itemize}
\item$\mathcal C_\alpha$ is a nonempty collection of club subsets of $\alpha$, with $| \mathcal C_\alpha |<\kappa$;
\item for every $C \in \mathcal C_\alpha$ and $\bar\alpha\in\acc(C)$, we have $C\cap\bar\alpha\in\mathcal C_{\bar\alpha}$;
\end{itemize}
\item For every cofinal $X\s\kappa$, there is $\alpha\in\acc(\kappa)$ such that $\sup(\nacc(C)\cap X)=\alpha$ for all $C \in \mathcal C_\alpha$.
\end{enumerate}
\end{defn}

In this paper, we shall not construct Souslin trees (we refer the reader to \cite{paper22} for background and definitions);
all we need is encapsulated in the following fact.
\begin{fact}[\cite{paper26}]\label{thefacts} $\boxtimes^*(\kappa)+\diamondsuit(\kappa)$ entails the existence of a $\kappa$-Souslin tree.
\end{fact}

\section{Proof of the Main Theorem}

\begin{defn}[\cite{paper29}] Suppose that $D$ is a club in $\kappa$. Define a function $\Phi_{D} : \mathcal P(\kappa) \to \mathcal P(\kappa)$ by letting,
for all $x\in\mathcal P(\kappa)$,
\[\Phi_{D}(x):=\begin{cases}
\{ \sup(D\cap\eta)\mid \eta\in x\ \&\ \eta>\min(D)\} &\text{if } \sup(D \cap \sup(x)) = \sup(x); \\
x \setminus \sup(D \cap \sup(x)) &\text{otherwise}.
\end{cases}\]
\end{defn}

\begin{lemma}\label{wide-club-guessing}
Suppose that:
\begin{itemize}
\item $\lambda<\kappa$;
\item $\vec{\mathcal C}=\langle \mathcal C_\alpha\mid\alpha<\kappa\rangle$ is a $\square(\kappa,\kappa)$-sequence;
\item $S$ is a stationary subset of $\{ \alpha\in \acc(\kappa)\mid |\mathcal C_\alpha|<\lambda\}$.
\end{itemize}

Then there exists a club $D\s\kappa$
such that for every club $E\s\kappa$,
there exists $\alpha\in S$ such that $\sup(\nacc(\Phi_D(C))\cap E)=\alpha$ for all $C\in\mathcal C_\alpha$.
\end{lemma}
\begin{remark}
All ingredients for the upcoming proof may already be found in \cite{paper29}.
For completeness, we give here a self-contained proof that avoids various concepts that appear in \cite{paper29}.
\end{remark}
\begin{proof}[Proof of Lemma]
Suppose not.
Then, for every club $D \subseteq \kappa$, we may find a club $E^D \subseteq \kappa$ such that, for every $\delta\in S$,
there is $C^D_\delta \in \mathcal C_\delta$ with $$\sup(\nacc(\Phi_D(C^D_\delta)) \cap E^D)<\delta.$$

Define a sequence $\langle E_i\mid i \leq \lambda \rangle$ of clubs in $\kappa$, by recursion, as follows:
\begin{itemize}
\item Set $E_0:=\kappa$;
\item For all $i< \lambda$, set $E_{i+1} := E^{E_i} \cap E_i$;
\item For all $i\in\acc(\lambda +1)$, set $E_i:=\bigcap_{j<i}E_j$.
\end{itemize}

Write $E:=E_{\lambda}$.
For each $\delta\in S$,
since $\{ C^{E_i}_\delta \mid i<\lambda \} \subseteq \mathcal C_\delta$, and
$\lambda=\cf(\lambda)>| \mathcal C_\delta |$,
we may pick $C_\delta\in\mathcal C_\delta$ such that
$I_\delta:=\{ i< \lambda \mid C^{E_i}_\delta=C_\delta\}$ is cofinal in $\lambda$.
Now, there are three cases to consider, each leading to a contradiction:

\underline{Case 1.} Suppose that there exists  $\delta\in S\cap E^{\kappa}_{>\omega}$ for which $\sup(E\cap\delta\setminus C_\delta)=\delta$.

Fix such $\delta$ and let $\{ i_n\mid n<\omega\}$ be the increasing enumeration of some subset of $I_\delta$.
Since $\langle E_i \mid i<\lambda\rangle$ is a $\s$-decreasing sequence,
for all $n<\omega$, we have in particular that $E_{i_{n+1}} \subseteq E_{i_n+1} \subseteq E^{E_{i_n}}$, so that $\alpha_n:=\sup(\nacc(\Phi_{E_{i_n}}(C_\delta))\cap E_{i_{n+1}})$ is $<\delta$. Put $\alpha:=\sup_{n<\omega}\alpha_n$.
As $\cf(\delta)>\omega$, we have $\alpha<\delta$.
Fix $\beta\in (E\cap\delta)\setminus C_\delta$ above $\alpha$.
Put $\gamma:=\min(C_\delta\setminus\beta)$.
Then $\delta>\gamma>\beta>\alpha$, and for all $i < \lambda$, since $\beta\in E\s E_i$,
we infer that $\sup( E_i\cap\gamma)\ge\beta$. So it follows from the definition of $\Phi_{E_i}(C_\delta)$ that $\min(\Phi_{E_i}(C_\delta)\setminus\beta)=\sup(E_i\cap \gamma)$ for all $i< \lambda$.
Since $\langle E_{i_n}\mid n<\omega\rangle$ is an infinite $\s$-decreasing sequence, let us fix some $n<\omega$
such that $\sup(E_{i_n}\cap\gamma)=\sup(E_{i_{n+1}}\cap\gamma)$.
Then $\min(\Phi_{E_{i_n}}(C_\delta)\setminus\beta)=\min(\Phi_{E_{i_{n+1}}}(C_\delta)\setminus\beta)$,
and in particular, $\beta^*:=\min(\Phi_{E_{i_n}}(C_\delta)\setminus\beta)$ is in $E_{i_{n+1}}\setminus(\alpha+1)$. Now, there are two options, each leading to a contradiction:
\begin{itemize}
\item[$\br$] If $\beta^*\in\nacc(\Phi_{E_{i_n}}(C_\delta))$, then we get a contradiction to the fact that  $\beta^*>\alpha\ge\alpha_n$.

\item[$\br$] If $\beta^*\in\acc(\Phi_{E_{i_n}}(C_\delta))$, then $\beta^*=\beta$ and $\beta^*\in\acc(C_\delta)$, contradicting the fact that $\beta\notin C_\delta$.
\end{itemize}

\underline{Case 2.} Suppose that there exists $\delta\in S\cap E^{\kappa}_{\omega}$ for which $\sup(E\cap\delta\setminus C_\delta)=\delta$.

Fix such $\delta$, and note that, for all $i\in I_\delta$, the ordinal $\alpha_i:=\sup(\nacc(\Phi_{E_{i}}(C_\delta))\cap E_{i+1})$ is $<\delta$.
So, as $\cf(\delta)\neq\omega_1$,
let $\{ i_\nu\mid \nu<\omega_1\}$ be the increasing enumeration of some subset of $I_\delta$, for which  $\alpha:=\sup_{\nu<\omega_1}\alpha_{i_\nu}$ is $<\delta$.
Fix $\beta\in (E\cap\delta)\setminus C_\delta$ above $\alpha$.
Put $\gamma:=\min(C_\delta\setminus\beta)$.
Then $\delta>\gamma>\beta>\alpha$, and $\min(\Phi_{E_i}(C_\delta)\setminus\beta)=\sup(E_i\cap \gamma)$ for all $i < \lambda$.
Fix some $\nu<\omega_1$
such that $\sup(E_{i_\nu}\cap\gamma)=\sup(E_{i_{\nu+1}}\cap\gamma)$.
Then  $\beta^*:=\min(\Phi_{E_{i_\nu}}(C_\delta)\setminus\beta)$ is in $E_{i_{\nu+1}}\setminus(\alpha+1)$, and as in the previous case, each of the two possible options leads to a contradiction.

\underline{Case 3.}  Suppose that $\sup(E\cap\delta\setminus C_\delta)<\delta$ for all $\delta\in S$.

Fix $\epsilon<\kappa$ for which $S':=\{\delta\in S\mid \sup(E\cap\delta\setminus C_\delta)=\epsilon\}$ is stationary.
Put $B:=\acc(E\setminus\epsilon)$, and note that, for every $\delta\in S'$, we have $B\cap\delta\s\acc(C_\delta)$.
Let $\{\beta_\alpha\mid \alpha<\kappa\}$ denote the increasing enumeration of the club $\{0\}\cup B$.
For all $\alpha<\kappa$, put:
 $$T_\alpha:=\{ C_\delta\cap\beta_\alpha\mid \delta\in S', \beta_\alpha<\delta\}.$$

\begin{claim}\label{claim1221} $(\bigcup_{\alpha<\kappa}T_\alpha,{\sq})$ is a tree whose $\alpha^{\text{th}}$ level is $T_\alpha$,
and $|T_\alpha|\le|\mathcal C_{\beta_\alpha}|$ for all $\alpha <\kappa$.
\end{claim}
\begin{proof}
We commence by pointing out that $T_\alpha\s  \mathcal C_{\beta_\alpha}$ for all $\alpha<\kappa$.
Clearly, $T_0=\{\emptyset\}=\mathcal C_0=\mathcal C_{\beta_0}$.
Thus, consider an arbitrary nonzero $\alpha<\kappa$ along with some $t \in T_\alpha$.
Fix $\delta \in S'$ above $\beta_\alpha$ such that $t = C_\delta \cap \beta_\alpha$.
Then $\beta_\alpha \in B \cap \delta \s\acc(C_\delta)$,
so that $C_\delta\cap\beta_\alpha\in \mathcal C_{\beta_\alpha}$. That is, $t\in\mathcal C_{\beta_\alpha}$.

This shows that for all $t\in \bigcup_{\alpha<\kappa}T_{\alpha}$: $$t\in T_{\alpha}\text{ iff }\sup(t)=\beta_{\alpha}.$$

Next, consider arbitrary $\alpha<\kappa$ and $t \in T_\alpha$,
and let $t_\downarrow := \{ s \in \bigcup_{\alpha'<\kappa}T_{\alpha'}\mid s \sq t, s \neq t \}$ be the set of predecessors of $t$.
Fix $\delta \in S'$ above $\beta_\alpha$ such that $t = C_\delta \cap \beta_\alpha$.
We claim that $t_\downarrow = \{ C_\delta \cap \beta_{\alpha'} \mid \alpha' < \alpha \}$,
from which it follows that $(t_\downarrow,{\sq})\cong(\alpha,{\in})$.

Consider $\alpha'<\alpha$.
Then $\beta_{\alpha'} < \beta_\alpha < \delta$,
so that $s := C_\delta \cap \beta_{\alpha'}$ is in $T_{\alpha'}$,
and it is clear that $s$ is a proper initial segment of $t$. That is, $s \in t_\downarrow$.

Conversely, consider $s \in t_\downarrow$.
Fix $\alpha'<\kappa$ such that $s\in T_{\alpha'}$.
By our earlier observation, $\sup(s) = \beta_{\alpha'}$,
so that since $s \sq t$, $s \neq t$, and $\sup(t) = \beta_\alpha$, we must have $\beta_{\alpha'} < \beta_\alpha$, and therefore $\alpha'<\alpha$.
Thus, $s = t\cap\beta_{\alpha'} = (C_\delta\cap\beta_\alpha) \cap \beta_{\alpha'} = C_\delta\cap \beta_{\alpha'}$, as required.
\end{proof}

Consider the stationary set $S'':=\{ \alpha\in S'\mid \alpha=\beta_\alpha\}$.
For each $\alpha\in S''$, we have $|T_\alpha|\le|\mathcal C_\alpha|<\lambda$,
so $T:=\bigcup_{\alpha\in S''}T_\alpha$ ordered by $\sq$ is a $\kappa$-tree each of whose levels has cardinality less than $\lambda$.
Now, by a lemma of Kurepa (see \cite[Proposition 7.9]{MR1994835}), $(T,\sq)$
admits a cofinal branch, i.e., a chain $C\s T$ (with respect to $\sq$)
that satisfies $|C\cap T_\alpha|=1$ for all $\alpha\in S''$.
Put $D := \bigcup C$ and note that $D$ is a club in $\kappa$.
As $\vec{\mathcal C}$ is a $\square(\kappa,\kappa)$-sequence,
let us pick $\beta\in\acc(D)$ such that $D\cap\beta\notin\mathcal C_\beta$.
Now, by definition of $D$, let us pick $t \in C$ such that $D\cap\beta\sq t$.
Then, as $t\in T$, let us pick $\delta\in S'$ above $\sup(t)$ such that $t \sq C_\delta$.
So $D \cap \beta \sq C_\delta$.
As $\beta \in \acc(D)$, we have $\beta\in\acc(C_\delta)$,  and hence $D\cap\beta=C_\delta\cap\beta\in\mathcal C_\beta$,
contradicting the choice of $\beta$.
\end{proof}

\begin{cor}\label{outcome}
Suppose that $\vec{\mathcal C}=\langle \mathcal C_\alpha\mid\alpha<\lambda^+\rangle$ is a $\square(\lambda^+,\lambda)$-sequence
for which $\{ \alpha\in E^{\lambda^+}_{<\lambda}\mid |\mathcal C_\alpha|<\lambda\}$ is stationary.
Then there exists a $\square(\lambda^+,\lambda)$-sequence $\vec{\mathcal C^\bullet}=\langle \mathcal C^\bullet_\alpha\mid\alpha<\lambda^+\rangle$
such that, for every club $E\s\lambda^+$,
there exists $\alpha\in \acc(\lambda^+)\cap E^{\lambda^+}_{<\lambda}$ with $|\mathcal C^\bullet_\alpha|<\lambda$ such that $\sup(\nacc(y)\cap E)=\alpha$ for all $y\in\mathcal C^\bullet_\alpha$.
\end{cor}
\begin{proof}
Appeal to Lemma~\ref{wide-club-guessing} with $\kappa:=\lambda^+$,
$\vec{\mathcal C}$, and
$S:=\{ \alpha\in \acc(\kappa)\cap E^\kappa_{<\lambda}\mid |\mathcal C_\alpha|<\lambda\}$,
and let $D\s\lambda^+$ be the outcome club.
Define $\vec{\mathcal C^\bullet}=\langle\mathcal C_\alpha^\bullet\mid\alpha<\lambda^+\rangle$ as follows:
\begin{itemize}
\item[$\br$] $\mathcal C^\bullet_0:=\{\emptyset\}$.

\item[$\br$] For all $\alpha<\lambda^+$, $\mathcal C^\bullet_{\alpha+1}:=\{\{\alpha\}\}$.

\item[$\br$] For all $\alpha\in\acc(\lambda^+)$, let $\mathcal C_\alpha^\bullet:=\{ \Phi_D(C)\mid C\in\mathcal C_\alpha\}$.
\end{itemize}

By \cite[Lemma~2.2]{paper29}, for all $\alpha\in\acc(\lambda^+)$ and $C\in\mathcal C_\alpha$,
$\Phi_D(C)$ is a club in $\alpha$
satisfying that, for all $\bar\alpha\in\acc(\Phi_D(C))$,  $\bar\alpha\in\acc(C)$ and $\Phi_D(C)\cap\bar\alpha=\Phi_D(C\cap\bar\alpha)\in\mathcal C^\bullet_{\bar\alpha}$.
In addition, by the choice of the club $D$, we know that for every club $E\s\lambda^+$,
there exists $\alpha\in \acc(\lambda^+)\cap E^{\lambda^+}_{<\lambda}$ with $|\mathcal C^\bullet_\alpha|\le|\mathcal C_\alpha|<\lambda$ such that $\sup(\nacc(y)\cap E)=\alpha$ for all $y\in\mathcal C^\bullet_\alpha$.

Finally, given an arbitrary club $D'$ in $\lambda^+$, consider the club $E:=\acc(D')$,
and fix $\alpha\in\acc(\lambda^+)$ such that $\sup(\nacc(y)\cap E)=\alpha$ for all $y\in\mathcal C^\bullet_\alpha$.
It follows that, for all $y\in\mathcal C_\alpha^\bullet$, $\nacc(y)\cap\acc(D'\cap\alpha)\neq\emptyset$, let alone $y\neq D'\cap\alpha$.
\end{proof}

We now arrive at the heart of the matter.

\begin{thm}\label{thm24}
Suppose that $\lambda^{<\lambda}=\lambda$, $2^\lambda=\lambda^+$,
and there exists a $\square(\lambda^+,\lambda)$-sequence $\langle \mathcal C_\alpha\mid\alpha<\lambda^+\rangle$
for which $\{ \alpha\in E^{\lambda^+}_{<\lambda}\mid |\mathcal C_\alpha|<\lambda\}$ is stationary.
Then $\boxtimes^*(\lambda^+)$ holds.
\end{thm}
\begin{proof} Let $\vec{\mathcal C^\bullet}=\langle \mathcal C^\bullet_\alpha\mid\alpha<\lambda^+\rangle$ be given by Corollary~\ref{outcome}.
Fix a bijection $\pi:\lambda^+\times\lambda\leftrightarrow\lambda^+$.
Also, for each $\beta<\lambda^+$, fix a bijection $g_\beta:\lambda\leftrightarrow E^{\beta+1}_{<\lambda}\times\lambda$.

By \cite[Lemma~2.1]{MR485361}, $\lambda^{<\lambda}=\lambda$ and $2^\lambda=\lambda^+$ imply together that $\diamondsuit^*(E^{\lambda^+}_{<\lambda})$ holds.
This means that we may fix a matrix $\langle Z_{\beta,j}\mid \beta\in E^{\lambda^+}_{<\lambda}, j<\lambda\rangle$ such that, for every $Z\s\lambda^+$,
for some club $D\s\lambda^+$, we have
$$D\cap E^{\lambda^+}_{<\lambda}\s \{\beta\in E^{\lambda^+}_{<\lambda}\mid \exists j<\lambda(Z\cap\beta=Z_{\beta,j})\}.$$

As $\lambda^{<\lambda}=\lambda$, the main result of \cite{ek} provides us with a sequence $\langle f_i\mid i<\lambda\rangle$
of functions from $\lambda^+$ to $\lambda$, such that, for every function $f:e\rightarrow\lambda$
with $e\in[\lambda^+]^{<\lambda}$, for some $i<\lambda$, we have $f\s f_i$.

Now, let $i<\lambda$ be arbitrary.
First, define a coloring $c_i:[\lambda^+]^2\rightarrow\lambda^+$ by letting, for all $\eta<\beta<\lambda^+$,
$$c_i(\eta,\beta):=\min\{\xi\in(\eta,\beta]\mid \xi=\beta\text{ or } \pi(\xi,i)\in Z_{g_\beta(f_i(\beta))}\}.$$
Then, for every $y\in\mathcal P(\lambda^+)$, let
$$y_i:=\acc(y)\cup\{ c_i(\sup(y\cap\beta),\beta)\mid \beta\in\nacc(y)\}.$$
Finally, for every $\alpha\in\acc(\lambda^+)$, let $\mathcal C_\alpha^i:=\{ y_i\mid y\in\mathcal C_\alpha^\bullet\}$.
Also, let $\mathcal C_0^i:=\{\emptyset\}$, and let $\mathcal C_{\alpha+1}^i:=\{\{\alpha\}\}$ for all $\alpha<\lambda^+$.

\begin{claim}\label{claim1} Suppose that $\alpha\in\acc(\lambda^+)$ and $C\in\mathcal C_\alpha^i$. Then:
\begin{enumerate}
\item $C$ is a club in $\alpha$;
\item For all $\bar\alpha\in\acc(C)$, $C\cap\bar\alpha\in\mathcal C^i_{\bar\alpha}$.
\end{enumerate}
\end{claim}
\begin{proof} Fix a club $y\in\mathcal C_\alpha^\bullet$ such that $C=y_i$.

(1) It is easy to see that for any two successive elements $\eta<\beta$ of the club $y$, we have that $C\cap (\eta,\beta]$ is a singleton.
Consequently, $\sup(C)=\sup(y)=\alpha$, and $\acc^+(C)\s\acc(y)$.
But, by definition of $C=y_i$, we also have $\acc(y)\s C$, so, $C$ is a club in $\alpha$.

(2) Let $\bar\alpha\in\acc(C)$ be arbitrary. By the above analysis, $\bar\alpha\in\acc(y)$, so that $y\cap\bar\alpha\in\mathcal C_{\bar\alpha}$.
But $C\cap\bar\alpha=y_i\cap\bar\alpha=(y\cap\bar\alpha)_i$, and hence $C\cap\bar\alpha\in\mathcal C_{\bar\alpha}^i$.
\end{proof}

\begin{claim} There exists $i<\lambda$ for which $\langle \mathcal C_\alpha^i\mid \alpha<\lambda^+\rangle$ witnesses $\boxtimes^*(\lambda^+)$.
\end{claim}
\begin{proof}
Suppose not. It follows from Claim~\ref{claim1} that for each $i<\lambda$,
we may pick some cofinal subset $X_i\s\lambda^+$ such that, for all $\alpha\in\acc(\lambda^+)$,
for some $C\in\mathcal C^i_{\alpha}$, we have $\sup(\nacc(C)\cap X_i)<\alpha$.

Let $Z:=\pi``\bigcup_{i<\lambda}(X_i\times \{i\})$,
and then fix a club $D$ in $\lambda^+$ such that for all $\beta\in D$:

\begin{itemize}
\item $\pi[\beta\times\lambda]=\beta$;
\item $\sup(X_i\cap\beta)=\beta$ for all $i<\lambda$;
\item if $\cf(\beta)<\lambda$, then there exists $j<\lambda$ with $Z\cap\beta=Z_{\beta,j}$.
\end{itemize}
Consider the club $E:=\acc(D)$.
By the choice of $\vec{\mathcal C^\bullet}$,
we may now pick $\alpha\in\acc(\lambda^+)\cap E^{\lambda^+}_{<\lambda}$ with $|\mathcal C^\bullet_\alpha|<\lambda$ such that $\sup(\nacc(y)\cap E)=\alpha$ for all $y\in\mathcal C^\bullet_\alpha$.
Since $\cf(\alpha)<\lambda$ and $|\mathcal C^\bullet_\alpha|<\lambda$, let us fix some $e\in[E\cap\alpha]^{<\lambda}$ such that $\sup(\nacc(y)\cap e)=\alpha$ for all $y\in\mathcal C^\bullet_\alpha$.
Define a function $f:e\rightarrow\lambda$ by letting, for all $\beta\in e$,
$$f(\beta):=\begin{cases}
\min g_\beta^{-1}\{ (\gamma,j)\in \{\beta\}\times \lambda \mid Z\cap\beta=Z_{\beta,j}\}\\\hfill\text{if }\cf(\beta)<\lambda;\\\\
\min g_\beta^{-1}\left\{ (\gamma,j)\in (D\cap E^\beta_{<\lambda})\times \lambda\Mid \begin{array}{c}Z\cap\gamma=Z_{\gamma,j}\text{, and for all }y\in\mathcal C^\bullet_\alpha,\\\beta\in\nacc(y)\implies\sup(y\cap\beta)<\gamma\end{array}\right\}\\\hfill\text{if }\cf(\beta)=\lambda.
\end{cases}$$
Fix $i<\lambda$ such that $f\s f_i$. By the choice of $X_i$, let us fix $C\in\mathcal C_\alpha^i$ such that $\sup(\nacc(C)\cap X_i)<\alpha$.
Find $y\in\mathcal C^\bullet_\alpha$ such that $C=y_i$.
Fix a large enough $\beta\in\nacc(y)\cap e$  such that $\eta:=\sup(y\cap\beta)$ is greater than $\sup(\nacc(C)\cap X_i)$.
In particular, $c_i(\eta,\beta)$ must be an element of $\nacc(C)\setminus X_i$.
Now, there are two cases to consider, each leading to a contradiction:

$\br$ If $\cf(\beta)<\lambda$, then for some $j<\lambda$, we have $g_\beta(f(\beta))=(\beta,j)$ and $Z\cap\beta=Z_{\beta,j}$.
But $\beta\in e\s E\s D$, so that  $g_\beta(f_i(\beta))=(\beta,j)$,
$\pi[\beta\times\lambda]=\beta$, and $$\{\xi<\beta\mid \pi(\xi,i)\in Z_{g_\beta(f_i(\beta))}\}=\{ \xi<\beta\mid \pi(\xi,i)\in Z\cap\beta\}=X_i\cap\beta.$$
As $\beta\in D$, we have $\sup(X_i\cap\beta)=\beta$,
so, $c_i(\eta,\beta)\in X_i\cap\beta$. This is a contradiction.

$\br$ If $\cf(\beta)=\lambda$, then let $(\gamma,j):=g_\beta(f(\beta))$, so that $\gamma\in D$ and $Z\cap\gamma=Z_{\gamma,j}$.
In particular,
$\{\xi<\gamma\mid \pi(\xi,i)\in Z_{g_\beta(f_i(\beta))}\}=X_i\cap\gamma$ and $\sup(X_i\cap\gamma)=\gamma$.
Since $\beta\in\nacc(y)$,
it also follows that $\eta=\sup(y\cap\beta)<\gamma<\beta$.
Consequently, $c_i(\eta,\beta)\in X_i\cap\beta$. This is a contradiction.
\end{proof}
This completes the proof.
\end{proof}

We are now ready to derive the Main Theorem.

\begin{proof}[Proof of the Main Theorem] By \cite[Lemma~2.1]{MR485361}, $\lambda^{<\lambda}=\lambda$ and $2^\lambda=\lambda^+$ imply together that $\diamondsuit^*(E^{\lambda^+}_{<\lambda})$ holds.
In particular, as $\lambda$ is uncountable, $\diamondsuit(\lambda^+)$ holds.
In addition, by Theorem~\ref{thm24}, $\boxtimes^*(\lambda^+)$ holds.
So, by Fact~\ref{thefacts}, there exists a $\lambda^+$-Souslin tree.\footnote{In fact, instead of appealing to Fact~\ref{thefacts}, we may appeal to a finer theorem from \cite{paper32},
thereby getting a $\lambda^+$-Souslin tree which is moreover $\lambda$-complete.}
\end{proof}

\end{document}